\documentclass[11pt]{amsart}

\usepackage{amssymb,amsmath,amscd}
\usepackage{bbm}
\usepackage{a4wide}

\theoremstyle{plain}
\newtheorem{theorem}{Theorem}
\newtheorem{prop}[theorem]{Proposition}
\newtheorem{lemma}[theorem]{Lemma}

\newtheorem{fact}[theorem]{Fact}

\theoremstyle{definition}
\newtheorem{remark}{Remark}

\newcommand{\ts}{\hspace{0.5pt}}
\newcommand{\nts}{\hspace{-0.5pt}}
\newcommand{\CC}{\mathbb{C}\ts}
\newcommand{\RR}{\mathbb{R}\ts}

\newcommand{\ZZ}{\mathbb{Z}}
\newcommand{\SSS}{\mathbb{S}}

\newcommand{\NN}{\mathbb{N}}
\newcommand{\XX}{\mathbb{X}}
\newcommand{\YY}{\mathbb{Y}}

\newcommand{\cB}{\mathcal{B}}
\newcommand{\cC}{\mathcal{C}}
\newcommand{\cH}{\mathcal{H}}

\newcommand{\cP}{\mathcal{P}}
\newcommand{\gL}{\varLambda}

\newcommand{\dd}{\, \mathrm{d}}

\begin{document}

\title[Close-packed dimers on the line]
{Close-packed dimers on the line: \\[2mm]
diffraction versus dynamical spectrum}

\author{Michael Baake}
\address{Fakult\"at f\"ur Mathematik, Universit\"at Bielefeld, \newline
\hspace*{\parindent}Postfach 100131, 33501 Bielefeld, Germany}
\email{mbaake@math.uni-bielefeld.de}

\author{Aernout van Enter}
\address{Johann Bernoulli Institute for Mathematics and
  Computer Science, \newline
 \hspace*{\parindent}University of Groningen, 
  PO Box 407, 9700\ts AK Groningen, The Netherlands}
\email{a.c.d.van.enter@rug.nl}

\begin{abstract}
  The translation action of $\RR^{d}$ on a translation bounded measure
  $\omega$ leads to an interesting class of dynamical systems, with a
  rather rich spectral theory.  In general, the diffraction spectrum
  of $\omega$, which is the carrier of the diffraction measure, lives
  on a subset of the dynamical spectrum. It is known that, under some
  mild assumptions, a pure point diffraction spectrum implies a pure
  point dynamical spectrum (the opposite implication always being
  true). For other systems, the diffraction spectrum can be a proper
  subset of the dynamical spectrum, as was pointed out for the
  Thue-Morse sequence (with singular continuous diffraction) in
  \cite{EM}.  Here, we construct a random system of close-packed
  dimers on the line that have some underlying long-range periodic
  order as well, and display the same type of phenomenon for a system
  with absolutely continuous spectrum. An interpretation in terms of
  `atomic' versus `molecular' spectrum suggests a way to come to a
  more general correspondence between these two types of spectra.
\end{abstract}

\maketitle

\section{Introduction}

It is well-known \cite{LMS,Martin,BL,LS} that pure point diffraction
and pure point dynamical spectrum, under some mild assumptions, are
equivalent properties of dynamical systems of translation bounded
measures on $d$-space.  This type of equivalence does not extend to
systems with continuous spectrum, as the example of the Thue-Morse
sequences shows \cite{EM}. The corresponding dynamical system can be
defined via the primitive $2$-letter substitution rule $a\mapsto ab$,
$b\mapsto ba$. It supports a unique shift-invariant probability
measure.  The diffraction spectrum (for the associated Dirac combs
with balanced weights $\pm 1$) is purely singular continuous, while
the dynamical spectrum has a non-trivial pure point part in form of
the dyadic rationals. This spectral information is not reflected in
the diffraction spectrum, no matter whether one works with balanced or
general weights.

However, this `missing' part can be extracted from the diffraction of
a factor of the Thue-Morse system, the so-called period doubling
sequences, which are Toeplitz sequences (the corresponding system can
be defined by the substitution rule $a\mapsto ab$, $b\mapsto aa$).
For a discussion of the dynamical spectra of both systems, we refer to
\cite{Q}, while the discrepancy with the diffraction spectrum was
noticed in \cite{EM}. The diffraction spectrum of the Thue-Morse and
period doubling systems (in modern terminology) is discussed in detail
in \cite{BG-TM,BG-PD}; see also the references given there for
previous work.

Below, we discuss a simple system that displays a similar phenomenon
in the presence of absolutely continuous diffraction.  We employ a
one-dimensional caricature of a system of dimeric molecules. It has an
absolutely continuous diffraction spectrum and an extra point in the
dynamical spectrum (for the $\ZZ$-action of the discrete shift), which 
is due to the presence of a superstructure
with periodic long-range order. One can recover this extra spectral
information by considering the diffraction of a \emph{factor} of the
original system. Our example illustrates the distinction between
dynamical and diffraction spectrum in a particularly simple manner.

Let us begin by briefly summarising the basic notions and concepts
from diffraction theory and dynamical systems (we refer to
\cite{Q,Robbie} and references therein for background on the dynamical
systems used here). In our context, it is best to use a
measure-theoretic setting for the systems under study, where we rely
on the Riesz-Markov representation theorem to identify regular Borel
measures on $\RR^{d}$ with (continuous) linear functionals on the
space $C_{\mathsf{c}} (\RR^{d})$ of continuous functions on $\RR^{d}$
with compact support. In particular, we consider unbounded, complex
measures $\omega$ on Euclidean space $\RR^{d}$ that are translation
bounded, which means that, for each compact set $K\subset \RR^{d}$, we
have $\sup_{t\in\RR^{d}} \lvert \omega \rvert (t+K) < \infty$; see
\cite{BF} for background.  Such an $\omega$ describes the realisation
of an infinite system, be it a crystal, a quasicrystal or a more
general object.

Given such an $\omega$, let $\widetilde{\omega}$ be obtained from
$\omega$ by reflection in the origin followed by complex conjugation,
so that $\widetilde{\omega} (g) = \overline{\omega (\widetilde{g})}$
for any continuous function $g$ of compact support, where
$\widetilde{g}$ is defined via $\widetilde{g} (x) = \overline{g
  (-x)}$. Given $\omega$, the corresponding autocorrelation measure
$\gamma = \gamma^{}_{\omega}$, or \emph{autocorrelation} for short, is
defined as the volume-averaged (or Eberlein \cite{Arga}) convolution
\begin{equation} \label{eq:def-auto}
   \gamma \, = \, \omega \circledast \widetilde{\omega}
    \, := \lim_{R\to\infty} \frac{\omega|^{}_{R} * 
      \widetilde{\omega|^{}_{R}}} {\mathrm{vol} 
      \bigl(B_{R} (0) \bigr)}\ts ,
\end{equation}
where $B_{R} (x)$ is the open ball of radius $R$ and centre $x$,
while $\omega|^{}_{R}$ denotes the restriction of $\omega$ to the
ball $B_{R} (0)$. The limit in Eq.~\eqref{eq:def-auto} is taken in the
vague topology, and will exist in all examples considered later,
at least almost surely in the probabilistic sense; for some general
results, see \cite{Hof}.

The measure $\gamma$ is positive definite by construction, and hence
Fourier transformable. This gives $\widehat{\gamma}$, the diffraction
measure or \emph{diffraction} for short, which is a positive measure
on $\RR^{d}$ that describes the outcome of kinematic diffraction
from $\omega$; see \cite{Cowley} for background and physical
applications. The diffraction has a unique decomposition as
\begin{equation} \label{eq:decomp}
   \widehat{\gamma} = \bigl(\widehat{\gamma}\bigr)_{\mathsf{pp}}
     + \bigl(\widehat{\gamma}\bigr)_{\mathsf{sc}}
     + \bigl(\widehat{\gamma}\bigr)_{\mathsf{ac}}
\end{equation}
into its pure point, singular continuous and absolutely continuous
parts, where the decomposition of the continuous part is relative to
Lebesgue measure. This is the Haar measure on $\RR^{d}$ and also the
right reference measure from the physical applications point of view.

A measure $\omega$ is called \emph{pure point diffractive} when the
corresponding diffraction measure satisfies $\widehat{\gamma} =
\bigl(\widehat{\gamma}\bigr)_{\mathsf{pp}}$, and similar definitions
apply to the other spectral components. Important examples for pure
point diffractive systems are perfect crystals and model sets
\cite{Hof,BM,BLM}, while the Thue-Morse sequence or the Rudin-Shapiro
sequence, both with balanced weights, are paradigms for systems with
purely singular continuous or purely absolutely continuous diffraction
spectra; see \cite{EM,BG-TM,R,S,HB,BG} and references therein for
more. Here and below, a point set $\gL\subset\RR^{d}$ is considered as
a measure on $\RR^{d}$ via its Dirac comb $\delta_{\gL} :=
\sum_{x\in\gL} \delta_{x}$, and a sequence $(w_{n})^{}_{n\in\ZZ}$ as a
measure on $\ZZ$ or on $\RR$ (or both) via the weighted Dirac comb
$\omega = \sum_{n\in\ZZ} w_{n}\ts \delta_{n}$, where $\delta_{x}$ is
the normalised point (or Dirac) measure at $x$.

In general, it is not adequate to restrict the attention to a single
measure $\omega$.  Equally relevant are translates of it, written as
$\delta_{t} * \omega$, or any other measure that can be approximated
arbitrarily well (in the vague topology) by such translates. Thus, one
defines the \emph{hull} $\XX_{\omega}$ of $\omega$ as its vague orbit
closure under the action of $\RR^{d}$,
\[
    \XX_{\omega} = \overline{ \{ \delta_{t} * \omega \mid 
    t \in \RR^{d} \} } \ts .
\]
It is clear that the $\RR^{d}$-action is continuous on $\XX_{\omega}$,
so that $( \XX_{\omega}, \RR^{d} )$ is a topological dynamical system.
Since $\omega$ is assumed to be translation bounded, the hull
$\XX_{\omega}$ is compact in the vague topology \cite{Martin}.
More generally, we will consider a compact space $\XX$ that contains
the orbit closure $\XX_{\omega}$ and emerges as the ensemble of
possible realisations of an ergodic stochastic process.

We equip $\XX$ with a translation invariant probability measure $\mu$
(which exists by standard arguments),
and consider the (measure-theoretic) dynamical system $( \XX, \cB, \mu
)$, where $\cB$ is the standard Borel $\sigma$-algebra induced by the
vague topology; see \cite{DGS} for background. The measure $\mu$ also
permits to consider $(\XX,\mu)$ as a stochastic process
\cite{EM,G,BBM}, as will be done below, too.  The action of $\RR^{d}$
now induces a unitary action on the Hilbert space $\mathrm{L}^{2}
(\XX, \mu)$. If the simultaneous eigenfunctions of the generators of
the $\RR^{d}$-action span $\mathrm{L}^{2} (\XX,\mu)$, one speaks of
\emph{pure point dynamical spectrum}. In general, as before, one can
have different spectral types, and one interesting question is the
relation between the diffraction and the dynamical spectrum.

What follows, is an attempt to improve this situation by way of some
guiding examples. Our focus will be on systems with some absolutely
continuous spectrum, as the interest in them has recently been on the
increase \cite{Wit}. Afterwards, we summarise some general observations
and formulate a more systematic connection between diffraction and
dynamical spectrum.

\section{A periodic toy model}\label{sec:toy}

All our examples below are built on closed subsets of $\{ \pm
1\}^{\ZZ}$, which is compact in the obvious product topology. Let us
begin with a quick glance at the set
\begin{equation} \label{eq:toy-def}
   \XX_{0} = \{ \ldots - + - | + - +  \ldots \, , \,
       \ldots + - + | - + - \ldots \}
\end{equation}
that consists of the two possible (truly) $2$-periodic sequences
within $\{ \pm 1\}^{\ZZ}$. Here and below, we use the shorthand $\pm$
for $\pm 1$ and write a bi-infinite sequence as $w = \ldots w^{}_{-2}
w^{}_{-1} | w^{}_{0} w^{}_{1} \ldots$, where $|$ denotes the
origin. Giving each element of $\ts\XX_{0}$ probability $1/2$ defines
$\mu^{}_{0}$, the only possible invariant probability measure on
$\XX_{0}$, which is thus ergodic. The corresponding dynamical system
$\bigl(\XX^{}_{0},\cP(\XX^{}_{0}),\mu^{}_{0}\bigr)$, with $\cP(A)$
denoting the power set of $A$, is clearly minimal, hence strictly
ergodic; compare \cite{W} for background.

\begin{remark} At this stage, we only consider the $\ZZ$-action
  induced by the discrete shift operator. Its suspension into a
  dynamical system under the action of the full translation group
  $\RR$ can later be added as a second step. It is trivial in the
  sense that one only sees the structure of the unit circle $\SSS$ in
  addition, in line with $\RR / \ZZ \simeq \SSS$. In contrast, the
  diffraction measures below are always measures on $\RR$.
\end{remark}

Considering any $ w \in \XX_{0}$ and turning it into a Dirac comb 
(viewed as a tempered measure on $\RR$) via
\[
   \omega \, = \sum_{m\in\ZZ} h^{}_{w(m)} \ts \delta_{m}
\]
with arbitrary complex weights $h^{}_{\pm}$, one quickly checks by
routine calculations that the corresponding autocorrelation in both
cases reads
\[
   \gamma \, = \,  \frac{\lvert h_{+} + h_{-} \rvert^{2}}{4} \, 
   \delta^{}_{\ZZ} + \frac{\lvert h_{+} - h_{-} \rvert^{2}}{4} 
   \ts \bigl( \delta^{}_{2\ZZ} - \delta^{}_{2\ZZ + 1} \bigr),
\]
with diffraction measure
\begin{equation}\label{eq:PSF}
   \widehat{\gamma} \, = \, \frac{\lvert h_{+} \nts + h_{-} \rvert^{2}}{4}
   \, \delta^{}_{\ZZ} + \frac{\lvert h_{+} \nts - h_{-} \rvert^{2}}{4}
   \, \delta^{}_{(2\ZZ+1)/2}\ts .
\end{equation}
Note that $\gamma$ and $\widehat{\gamma}$ are to be understood as
measures on $\RR$.
The last formula follows from an application of the Poisson summation
formula for lattice Dirac combs \cite{BM}, which reads
\[
     \widehat{\delta^{}_{\!\varGamma}} \, = \,
     \mathrm{dens} (\varGamma) \, \delta^{}_{\!\varGamma^{*}}
\]
for a lattice $\varGamma$ and its dual lattice $\varGamma^{*}$.
As expected for a periodic structure, $\widehat{\gamma}$ is a pure
point measure.  Since the original measure $\omega$ is supported on
$\ZZ$, its diffraction is $1$-periodic \cite{B}. The $2$-periodicity
of $\omega$ in turn results in $\ZZ/2$ as the support of
$\widehat{\gamma}$.  The latter can alternatively be written as
\[
    \widehat{\gamma} \, = \, \frac{1}{4} \Bigl(
    \lvert h_{+} \nts + h_{-}\rvert^{2}\, \delta^{}_{0} +
    \lvert h_{+} \nts - h_{-}\rvert^{2}\, \delta^{}_{1/2} \Bigr)
    * \delta^{}_{\ZZ}\ts ,
\]
which illustrates both aspects. In particular, when $h^{}_{\pm}=\pm 1$,
one obtains $\widehat{\gamma} = \delta^{}_{\ZZ + \frac{1}{2}}$.

Let us compare this with the dynamical spectrum (under the
$\ZZ$-action of the shift). Here, we clearly have $\mathrm{L}^{2}
(\XX^{}_{0}, \mu^{}_{0}) \simeq \CC^{2}$, and there are two
eigenfunctions, $f\equiv 1$ (for eigenvalue $1$) and $g$ (for
eigenvalue $-1$), the latter defined by $w\mapsto g(w) = w^{}_{0}$.
Together, they form an orthonormal basis of $\CC^{2}$, relative to the
inner product $\langle f | g \rangle = \int_{\XX_{0}} \bar{f} g \dd
\mu$. So, this is a simple example where one clearly sees how pure
point diffraction spectrum and pure point dynamical spectrum fit
together (and are, in fact, equivalent \cite{LMS,BL,LS}). We will
explain this in more detail later on.

\section{Close-packed dimers on the line with random 
orientation}\label{sec:dimers}

Consider close-packed dimers on the integers, at this stage viewed as
empty boxes of length $2$ that cover $\ZZ$ without overlaps or gaps.
There are two possible configurations, which are distinguished by the
central box (the latter either occupying the positions $0$ and $1$, or
$-1$ and $0$). Let us now fill the boxes randomly with dimeric
`molecules', by distributing weights $\pm 1$ such that each dimer
carries a $1$ and a $-1$, but in random order (or orientation).  So,
each box is then either $(+,-)$ or $(-,+)$. The result is a sequence
in $\{ \pm 1 \}^{\ZZ}$, where we now disregard the boxes again (they
can always be reconstructed from a given sequence unless it is one of
the two periodic sequences from $\XX_{0}$).

The ensemble of all dimeric sequences as described above forms a
closed and compact shift space $\XX$, with a continuous action of the
group $\ZZ$ via the usual shift operation. We call it the
\emph{dimeric molecule shift}, or DMS for short.  More precisely, the
shift $S \! : \, \XX \longrightarrow \XX$, defined by $(Sw)_{n} =
w_{n+1}$, is a continuous automorphism on $\XX$ and generates 
an action of $\ZZ$. We will come back to this point of view shortly.

Consider a sequence $w\in\XX$ and form the corresponding weighted Dirac
comb
\begin{equation} \label{eq:bal-comb}
    \omega \, = \, w \, \delta^{}_{\ZZ}
    \, := \sum_{m\in\ZZ} w_{m} \delta_{m} \ts ,
\end{equation}
which is a translation bounded (signed) measure on $\ZZ$ (and also on
$\RR$, via the canonical embedding of $\ZZ$ in $\RR$). The
corresponding autocorrelation measure (if it exists) is of the form
$\gamma = \sum_{n\in\ZZ} \eta(n) \ts \delta_{n}$, where the
coefficients are given by the limits
\begin{equation} \label{eq:eta-coeff}
   \eta(n)\, = \lim_{N\to\infty} \frac{1}{2N+1}
   \sum_{m=-N}^{N} \overline{w_{m}}\, w_{m+n} \, ,
\end{equation}
provided the latter exist.

\begin{lemma} \label{lem:dimer-auto}
  For any $n\in\ZZ$, the autocorrelation coefficient $\eta(n)$
  of the close-packed dimer model on $\ZZ$ with random orientation
  is given by
\[
    \eta (n) = \begin{cases} 1, & n=0 \\
       -\frac{1}{2}, & n=\pm 1 \\
       0, & \text{otherwise}. \end{cases}
\]
  Consequently, the autocorrelation is\/ $\gamma = \delta^{}_{0}
  - \frac{1}{2} (\delta^{}_{1} + \delta^{}_{-1})$, which applies
  to almost all realisations of the DMS process.
\end{lemma}
\begin{proof}
  The process is clearly stationary and ergodic, so that we can
  determine $\eta$ from a typical realisation. One has $\eta(0)=1$ for
  every realisation, and $\eta(n) = \eta(-n)$ for $n\in\ZZ$ is clear
  whenever one of the coefficients exist. So, it remains to show
  the claim for $n\in\NN$.

  Since the sequences with all dimers in the same orientation form a
  null set, we may assume that at least one position $i\in\ZZ$ exists
  such that $w_{i} = w_{i+1}$.  In a typical realisation, we
  have infinitely many such positions in the sequence, and they are
  either all even or all odd. We will use this structure implicitly
  in what follows.

  When $n\ge 2$, the values $w_{m}$ and $w_{m+n}$ are independent.
  Moreover, the sum in Eq.~\eqref{eq:eta-coeff} can be split into four
  sums, each of which is a sum over i.i.d.\ random variables of
  Bernoulli type. The strong law of large numbers (SLLN; see \cite{Ete}
  for instance for a formulation with the slightly weaker assumption of 
  mere pairwise independence, which will come in handy later on)
  then tells us that each contribution almost surely vanishes, so that
  $\eta(n)=0$ in this case.
  
  In the remaining case ($n=1$), every second term is $-1$ due to
  the structure of the dimers, which sums to $-\frac{1}{2}$. The
  remaining terms are the ones that cross the dimer boundaries,
  hence contribute $1$ or $-1$ with equal probability and thus
  (almost surely) do not contribute to the overall sum (again by
  the SLLN). This gives $\eta(1)=-\frac{1}{2}$ and the proof is
  complete.
\end{proof}

The corresponding diffraction follows by a straight-forward
calculation.

\begin{prop} \label{prop:bal-diff}
  The diffraction of the DMS model is given by
\[ 
    \widehat{\gamma} = ( 1 - c ) \lambda
\]
  with $c(k) = \cos(2 \pi k)$, so that the measure $\widehat{\gamma}$
  is absolutely continuous with respect to Lebesgue measure $\lambda$.
  \qed
\end{prop}

This particularly simple result is due to the balanced nature
of the weights, so that
\begin{equation} \label{eq:balanced}
     \lim_{N\to\infty} \frac{1}{2N\nts +1} \sum_{n=-N}^{N} w_{n} = 0
\end{equation}
holds for all realisations of our process. Let us now consider
general weights, which we realise via the mapping $h \!:\, \{ \pm 1\}
\longrightarrow \CC$ that takes values $h_{\pm}$. Given a realisation
$w$, the new Dirac comb is then
\begin{equation} \label{eq:new-Dirac}
    \omega^{}_{h} \, = \sum_{m\in \ZZ} h(w_{m}) \, \delta_{m}
   \, = \, \frac{h_{+} + h_{-}}{2}\, \delta^{}_{\ZZ} +
      \frac{h_{+} - h_{-}}{2}\, \omega \ts , 
\end{equation}
where $\omega$ is the Dirac comb of Eq.~\eqref{eq:bal-comb}.
The new autocorrelation turns out to be
\begin{equation} \label{eq:new-auto}
   \gamma^{}_{h} = \omega^{}_{h} \circledast 
         \widetilde{\omega^{}_{h}} =
   \frac{\lvert h_{+} + h_{-}\rvert^{2}}{4}\, 
         \delta^{}_{\ZZ} +
   \frac{\lvert h_{+} - h_{-}\rvert^{2}}{4}\, \gamma
\end{equation}
with the $\gamma$ from the balanced weight case of
Lemma~\ref{lem:dimer-auto}. Note that Eq.~\eqref{eq:new-auto} 
holds almost surely (as $\gamma$ does), and rests upon  
$\delta^{}_{\ZZ}\circledast\widetilde{\delta^{}_{\ZZ}} =
\delta^{}_{\ZZ}$ together with
$\delta^{}_{\ZZ}\circledast \widetilde{\omega} = 0$ and
$\widetilde{\delta^{}_{\ZZ}} \circledast \omega = 0$. The latter two
identities are a consequence of Eq.~\eqref{eq:balanced}.

The corresponding diffraction measure $\widehat{\gamma^{}_{h}}$ can be
calculated via Fourier transform and an application of
Proposition~\ref{prop:bal-diff}.  It (almost surely) reads
\begin{equation} \label{eq:new-diff}
   \widehat{\gamma^{}_{h}} = 
   \frac{\lvert h_{+} + h_{-}\rvert^{2}}{4}\, \delta^{}_{\ZZ} +  
   \frac{\lvert h_{+} - h_{-}\rvert^{2}}{4}\,
   (1 - c) \lambda \ts ,
\end{equation}
with $c(k) = \cos (2 \pi k)$ as above.  This is a measure of mixed
type, with a pure point part and an absolutely continuous
one. However, the point part is trivial in the sense that it only
reflects the lattice support of the Dirac comb $\omega^{}_{h}$ and
thus does not carry any other relevant information on the system. This
corresponds to the trivial (constant) eigenfunction of the dynamical
spectrum, which we determine next.

\section{The DMS and its dynamical spectrum}

Let us look at the above model from the viewpoint of dynamical
systems. As before, we begin with the $\ZZ$-action of the shift operator.
Given a sequence $w\in \{ \pm 1\}^{\ZZ}$, let us first define
\begin{equation} \label{eq:borders}
   M(w) = \{ m\in\ZZ \mid w_{m} = w_{m+1} \}\ts .
\end{equation}
Note that $M(w) = \varnothing$ precisely for $w\in\XX_{0}$, with
$\XX_{0}$ from Eq.~\eqref{eq:toy-def}.  The set of all dimeric
molecule sequences from Section~\ref{sec:dimers} forms the ensemble
\begin{equation} \label{es:dms-def}
   \XX = \{ w \in \{ \pm 1 \}^{\ZZ} \mid
   M(w) \subset 2\ZZ \text{ or } M(w) \subset 2\ZZ+1 \}\ts ,
\end{equation}
which is a closed subshift (and hence a compact set). We call
a sequence even (odd) when $M(w)$ is non-empty and a subset
of $2\ZZ$ (of $2\ZZ +1$). Then, $\XX$ splits as 
\[
    \XX = \XX_{+} \,\dot{\cup} \,\ts \XX_{-} 
          \,\dot{\cup}\,\ts \XX_{0}\ts ,
\]
where $\XX_{+}$ and $\XX_{-}$ denote the closed subsets of even and odd
sequences, respectively, while $\dot{\cup}$ denotes the disjoint union
of sets.

The shift $S$ (as defined earlier) is viewed as the (continuous)
generator of the action of $\ZZ$ on $\{ \pm 1\}^{\ZZ}$.  As $\XX$ is
clearly shift-invariant, we obtain $(\XX,S)$ as a topological
dynamical system (with $\ZZ$-action). It is clear from standard
arguments that there are invariant probability measures on $\XX$.
Indeed, the underlying process highlights a natural choice for a
measure $\mu$, as is also clear from the proof of
Lemma~\ref{lem:dimer-auto}. It satisfies $\mu (\XX^{}_{0})=0$ together
with $\mu (\XX_{+})=\mu (\XX_{-})=\frac{1}{2}$. Within $\XX_{+}$, each
dimer then has equal probability to be either $(+,-)$ or $(-,+)$, so
that the corresponding cylinder sets and their measures are
well-defined.  We can now view $(\XX, \cB^{}_{\XX}, \mu)$ as a
measure-theoretic dynamical system (under the action of $\ZZ$ via the
shift $S$), where $\cB^{}_{\XX}$ is the standard Borel
$\sigma$-algebra on $\XX$.

Let us next consider the Hilbert space $\cH = \mathrm{L}^{2}
(\XX,\mu)$, with the induced action of $S$ via $U \! : \, \cH
\longrightarrow \cH$, as defined by $f \mapsto Uf$ with $Uf (w) :=
f(Sw)$. The inner product is written as
\[
   \langle f \mid g \rangle = \int_{\XX} \overline{f(w)}\, 
   g(w)  \dd \mu (w) \ts ,
\]
where $\langle Uf \mid Ug \rangle = \langle f \mid g \rangle$ holds
due to the shift invariance of $\mu$. In fact, $U$ is unitary.

The function $\varphi \equiv 1$ is an eigenfunction of $U$ with
eigenvalue $1$ as usual, but we also have an eigenfunction for the
eigenvalue $-1$, namely the one defined by
\begin{equation} \label{eq:eigenfun}
   \psi(w) = \begin{cases} 0, & \text{if } w\in \XX_{0} , \\
          \pm 1, & \text{if } w\in \XX_{\pm}, \end{cases}
\end{equation}
which is well-defined because $\XX_{0}$, $\XX_{+}$ and $\XX_{-}$ are
measurable sets. Clearly, $\langle \varphi \mid \psi \rangle = 0$,
since $\mu(\XX_{+}) = \mu (\XX_{-}) = \frac{1}{2}$, while $\psi^{2} =
\varphi$ holds $\mu$-almost everywhere. Note also that $\psi$ can be
written as a limit via
\[
    \psi (w) \, = \lim_{N\to\infty} \frac{2}{2N+1}
    \sum_{n=-N}^{N} (-1)^{n}\ts  w_{n} \, w_{n+1} \ts ,
\]
which exists for $\mu$-almost all $w\in\XX$ by an SLLN argument  
analogous to that used in the proof of Lemma~\ref{lem:dimer-auto}.

So far, we know that the $\CC$-span of $\varphi$ and $\psi$ is
contained in $\cH_{\mathsf{pp}}$.  To see that we actually have equality
here, consider the double shift $S^{2}$ on $\XX$. First, for any $w\in
\XX$, the entries $w_{n}$ and $w_{n+2}$ are independent, so that, for
any fixed $i\in \ZZ$, the sequence $(S^{2n}w)_{i}$ with $n\in\ZZ$ is a
coin tossing sequence. Consequently, $S^{2}$ must comprise a spectrum
of countable Lebesgue type. On the other hand, for any given $w$, the
sequences with indices $i$ and $i+1$ are dependent, which is reflected
by the fact that $\XX$ is not minimal for the action of $S^{2}$, Here,
$\XX_{0}$, $\XX_{+}$ and $\XX_{-}$ are the non-trivial invariant
subspaces. They lead to two eigenfunctions of $S^{2}$ with eigenvalue
$1$, namely the characteristic functions $\mathbf{1}_{+}$ and
$\mathbf{1}_{-}$ (on $\XX_{+}$ and $\XX_{-}$), with $\varphi
=\mathbf{1}_{+} + \mathbf{1}_{-}$ and $\psi= \mathbf{1}_{+} -
\mathbf{1}_{-}$ (both holding $\mu$-almost everywhere). Our
system is one-dependent in the sense of \cite{AGKV}, and can be
viewed as the average of two Bernoulli shifts. Consequently,
$1$ is the only eigenvalue of $S^{2}$, and no singular continuous
contribution exists. The spectral theorem now tells us 
that $S$ has precisely the two eigenfunctions constructed above and 
only absolutely continuous spectrum otherwise. In particular, we have 
no freedom for singular continuous components.

Let us now expand on the continuous part of the spectrum.
To this end, we consider the function defined by
$w\mapsto \sigma^{}_{n} (w) = w_{n} + w_{n+1}$, which is continuous
on $\XX$ and hence measurable. A short calculation reveals that
\[
    \langle \varphi \mid \sigma^{}_{n} \rangle =
    \langle \psi \mid \sigma^{}_{n} \rangle = 0\ts ,
\]
so that $\sigma^{}_{n} \in \cH^{\perp}_{\mathsf{pp}}$,
while $U\sigma^{}_{n} = \sigma^{}_{n+1}$. The smallest $U$-invariant
subspace of $\cH$ that contains $\sigma^{}_{n}$ is thus the cyclic
space
\[
    \cC (\sigma^{}_{n}) =  \bigoplus_{m\in\ZZ} \CC \, \sigma^{}_{m}\ts .
\]

For any $m\in\ZZ$, the spectral measure of $\sigma^{}_{m}$ is given by
\begin{equation} \label{es:spec-meas}
   \langle \sigma^{}_{m} \mid U^{n} \sigma^{}_{m} \rangle =
   \! \int_{\XX} (w_{m} + w_{m+1})
   (w_{m+n} + w_{m+n+1}) \dd \mu (w) =
   \delta^{}_{n,0} - \frac{1}{2} 
   \bigl( \delta^{}_{n,2} + \delta^{}_{n,-2} \bigr),
\end{equation}
which is a positive definite function on $\ZZ$. In its calculation,
we have used
\[
    \int_{\XX} w_{m} w_{m+k} \dd \mu (w) = \begin{cases}
    1, & \text{if } k=0, \\ -\frac{1}{2}, & \text{if }
    \lvert k \rvert = 1, \\ 0, & \text{if } \lvert k \rvert \ge 2 ,
    \end{cases}
\]
which follows by an argument used before for the determination of the
autocorrelation coefficients $\eta (m)$ in Lemma~\ref{lem:dimer-auto}.

By the Herglotz-Bochner theorem, the positive definite function $n
\mapsto \langle \sigma^{}_{m} \mid U^{n} \sigma^{}_{m} \rangle$ is the
Fourier transform of a positive measure, which means that
\[
   \langle \sigma^{}_{m} \mid U^{n} \sigma^{}_{m} \rangle =
   \int_{0}^{1} e^{-2 \pi i n x} \dd \nu^{}_{m} (x)
\]
for some positive measure $\nu^{}_{m}$ on the unit circle $\SSS$, 
represented here by the unit interval (with periodic boundary 
conditions).  By routine calculation, one finds
\begin{equation} \label{eq:pos-meas}
   \nu^{}_{m} = \lambda - \frac{1}{2} \bigl(
   e^{4 \pi i k} + e^{- 4 \pi i k} \bigr) \lambda =
   \bigl( 1 - \cos(4 \pi k) \bigr) \lambda \ts ,
\end{equation}
where the bracketed factor is the Radon-Nikodym density of
$\nu^{}_{m}$ relative to Lebesgue measure $\lambda$, written as a
function of the variable $k$ (which is equivalent to our previous
formulation of a density, as used in Proposition~\ref{prop:bal-diff}
and Eq.~\eqref{eq:new-diff}). In particular, $\nu^{}_{m}$ is
absolutely continuous (relative to $\lambda$), and does not depend on
$m$. Note that $\nu^{}_{m}$ is obtained from the diffraction measure
by doubling the argument in its Radon-Nikodym density.

This is one of countably many mutually orthogonal cyclic spaces with
absolutely continuous spectral measures, as once again follows from
the underlying Bernoulli structure. Together with the two eigenvalues
and the absence of a singular continuous part, we have established
the following result.

\begin{prop}\label{prop:dms-dynamical}
  The dynamical spectrum of the DMS under $\ZZ$-action is a mixture of
  a pure point part and an absolutely continuous one. The eigenvalues
  are\/ $\pm 1$, while the remainder is of countable Lebesgue type.
  \qed
\end{prop}

So far, we have formulated the spectrum `naively', without any
reference to harmonic analysis and duality. It is more systematic
to include the dual group to $\ZZ$ into the picture, which is the
unit circle $\SSS$, conveniently represented by the half-open interval
$[0,1)$ together with addition modulo $1$. Then, one sees (via
the elements of the dual group as characters on $\ZZ$) that our
eigenvalues $1$ and $-1$ correspond to the elements $0$ and 
$\frac{1}{2}$ of the unit circle, with $\frac{1}{2}$ being the
non-trivial contribution.

This point of view is particularly useful when we suspend the
$\ZZ$-action into the continuous translation action of the group
$\RR$, as given by $\omega \mapsto \delta_{t} * \omega$ with
$t\in\RR$. Now, the dual group is $\RR$ (since $\RR$ is self-dual),
and our dynamical spectrum becomes $\ZZ/2$.  This is the natural
formulation for the comparison with the (support of the) diffraction
spectrum. The pure point part of the dynamical spectrum is non-trivial
because it is $\ZZ/2$ rather than $\ZZ$.

At this point, we note that the relation between the dynamical and the
diffraction spectrum is reminiscent of the situation for the
Thue-Morse sequence. In both cases, there exists some non-trivial
point spectrum that is not reflected in the diffraction measure of the
sequence. However, the missing spectral part of the Thue-Morse
sequence is retrieved via the period doubling sequence, which (as a
dynamical system) can be viewed as a factor the the Thue-Morse system.
We will now demonstrate that the analogous situation is also met for
our new example.

\section{A factor system and its diffraction}

Let us define a mapping $ \phi \! : \, \XX \longrightarrow \{ \pm 1
\}^{\ZZ}$ via $w \mapsto \phi(w)$ with 
\[
    \phi(w)^{}_{n} = -\ts w^{}_{n} w^{}_{n+1} \ts ,
\] 
which is continuous. This particular mapping is inspired by the
analogous situation for the Thue-Morse sequence, and indeed has
similar consequences here.  The image set, $\YY = \phi(\XX)$, is again
compact, and $(\YY, S)$ is another topological dynamical system. It is
clear that $\phi(-w) = \phi(w)$, where $(-w)_{n} = -w_{n}$, and a
moments reflection shows that this is the only ambiguity, so that
$\phi \! : \, \XX \longrightarrow \YY$ is a globally two-to-one
surjection. In particular, $\YY_{0} = \phi (\XX_{0}) = \{ \ldots 11 |
11 \ldots \}$, and the entire image shift space is
\[
     \YY = \big\{ v \in \{ \pm 1 \}^{\ZZ} \mid
     v_{n} = 1 \text{ for all } n \in 2 \ZZ
     \text{ or for all } n \in 2 \ZZ + 1 \big\} .
\]
Moreover, $\phi$ makes the diagram
\begin{equation} \label{eq:diagram}
   \begin{CD}
    \XX   @>S>> \XX \\
    @V\phi VV @VV\phi V \\
    \YY   @>S>> \YY
   \end{CD}
\end{equation}
commutative. Consequently, $(\YY, S)$ is a (topological) factor
\cite{DGS} of the dynamical system $(\XX, S)$. In our setting, the
measure $\mu$ on $\XX$ induces a measure $\nu$ on $\YY$ via $\nu (A) =
(\phi . \mu) (A) := \mu ( \phi^{-1} (A))$ for Borel sets $A$.  We may
thus also consider the dynamical system $(\YY,\cB^{}_{\YY},\nu)$,
which is then a measure-theoretic factor of $(\XX,\cB^{}_{\XX},\mu)$;
see \cite[Sec.~3]{BL-2} for a summary of factors and their spectral
properties.

Let us first look at diffraction, for a typical element $v\in\YY$. 

\begin{lemma} \label{lem:factor-auto} 
  The autocorrelation coefficients of $v\in\YY$ are $\nu$-almost
  surely given by $\eta(0)=1$, $\eta(2n) = \frac{1}{2}$ for all
  $n\in\ZZ \setminus \{0\}$, and $\eta(2m+1)=0$ for all $m\in\ZZ$.
\end{lemma}

\begin{proof}
  The mapping $\phi$ has the effect that a typical $v\in\YY$ has
  weights $1$ on every second position, and weights $\pm 1$ with equal
  probability on all remaining positions. The latter form an i.i.d.\
  family of random variables, so that an application of the SLLN, in
  the same spirit as used before, gives the formula for $\eta$.
\end{proof}

\begin{prop}\label{prop:factor-diff}
  Autocorrelation and diffraction of\/ $\YY$ are given by
\[
   \gamma \, = \, \frac{1}{2} \, \delta^{}_{0} +
   \frac{1}{2} \, \delta^{}_{2\ZZ}
   \quad \text{and} \quad
   \widehat{\gamma} \, = \, \frac{1}{2} \, \lambda +
   \frac{1}{4} \, \delta^{}_{\ZZ/2} \ts ,
\]
  which apply to almost all realisations of the underlying process.

  In particular, the diffraction is of mixed type, with a
  non-trivial pure point component.
\end{prop}

\begin{proof}
  The claim on $\gamma$ is clear from Lemma~\ref{lem:factor-auto},
  while its Fourier transform follows from $\widehat{\delta_{0}} =
  \lambda$ together with an application of the Poisson summation
  formula to $\delta^{}_{2\ZZ}$.
\end{proof}

As before, this is the result for weights $\pm 1$. Since they are no
longer balanced, the calculation of the diffraction formula for
general weights needs one extra step. Observing that
\[
   \lim_{N\to\infty} \frac{1}{2N \nts +1}
   \sum_{n=-N}^{N} v^{}_{n} \, = \, \frac{1}{2}
\]
holds for $\nu$-almost all $v\in\YY$, one easily derives the
Eberlein convolutions
\[
    \delta^{}_{\ZZ} \circledast \widetilde{\omega} \, = \, 
    \frac{1}{2} \, \delta^{}_{\ZZ}
    \quad \text{and} \quad
    \omega \circledast \widetilde{\delta^{}_{\ZZ}} \, = \, 
    \frac{1}{2} \, \delta^{}_{\ZZ} \ts ,
\]
which apply $\nu$-almost surely.  Since the general Dirac comb
$\omega^{}_{h}$ again satisfies Eq.~\eqref{eq:new-Dirac}, a simple
calculation results in
\[
    \gamma^{}_{h} \, = \, 
    \frac{\lvert h_{+} \nts + h_{-}\rvert^{2} + \lvert h_{+} \rvert^{2} 
    - \lvert h_{-} \rvert^{2}}{4} \, \delta^{}_{\ZZ}
    + \frac{\lvert h_{+} \nts - h_{-}\rvert^{2}}{4} \, \gamma
\]
and thus in the general diffraction formula
\begin{equation} \label{eq:new-gen-diff}
    \widehat{\gamma^{}_{h}} \, = \,
    \frac{\lvert h_{+} \nts + h_{-}\rvert^{2} + \lvert h_{+} \rvert^{2} 
    - \lvert h_{-} \rvert^{2}}{4} \, \delta^{}_{\ZZ} 
    + \frac{\lvert h_{+} \nts - h_{-}\rvert^{2}}{16} \, \delta^{}_{\ZZ/2}
    + \frac{\lvert h_{+} \nts - h_{-}\rvert^{2}}{8} \, \lambda \ts ,
\end{equation}
by an application of Proposition~\ref{prop:factor-diff}.  The
intensity of any point measure $\delta_{k}$ with $k\in\ZZ$ is thus
given by $\lvert \frac{3}{4} h_{+} + \frac{1}{4} h_{-} \rvert^{2}$,
which is the absolute square of the average weight (or scattering
strength) in this case, as it must. Eq.~\eqref{eq:new-gen-diff}
displays a non-trivial pure point component (namely the one with
support $\ZZ/2$) that `recovers' the missing part from our original 
dynamical system $\bigl( \XX, \cB^{}_{\XX}, \mu\bigr)$.

The dynamical spectrum of $\bigl( \YY, \cB^{}_{\YY}, \nu\bigr)$ is the
same as that of $\bigl( \XX, \cB^{}_{\XX}, \mu\bigr)$. Given an element
$v\in\YY$, our previous eigenfunction $g$ takes the same value on the
two pre-images in $\phi^{-1} (v)$, so that we once again have an
eigenfunction for the eigenvalue $-1$, in addition to the constant
function (for the eigenvalue $1$). The Bernoulli substructure is
responsible for a continuous spectrum of countable Lebesgue type.
Note that a factor system need not have the same dynamical spectrum 
as the original system, as the example of the Thue-Morse sequence 
versus the period doubling sequence demonstrates.

\begin{remark}\label{rem:simpler}
  There is another very simple (albeit somewhat degenerate)
  possibility to define a factor. Recall that $\XX_{0} = \{ u^{}_{+},
  u^{}_{-} \}$ with $u^{}_{+} = \ldots +- | +- \ldots$ and $u^{}_{-} =
  Su^{}_{+}$. Now, define the mapping $\psi \! : \, \XX
  \longrightarrow \XX_{0}$ by
\[
    w \,\longmapsto\, \psi (w) := \begin{cases}
    w , &  \text{if $w\in \XX_{0}$}, \\
    u^{}_{\pm} , &  \text{if $w\in \XX_{\pm}$} ,
    \end{cases} 
\]
which is a continuous surjection and shows that our toy system of
Section~\ref{sec:toy} is a factor of the DMS. Here, the factor
has pure point spectrum. The diffraction spectrum is $\ZZ/2$,
which exhausts the dynamical spectrum of the DMS.
\end{remark}

Our observations up to this point can be summarised as follows.
\begin{theorem}
   The diffraction measure of the DMS system with balanced weights is
   purely absolutely continuous, as stated in
   Proposition~$\ref{prop:bal-diff}$.  The case of general weights can
   only add the trivial pure point part, as given in
   Eq.~\eqref{eq:new-diff}.
    
    The dynamical spectrum of the DMS system under the action of\/
    $\RR$ consists of the pure point part\/ $\ZZ/2$ together with a
    countable Lebesgue spectrum.
    
    The non-trivial part\/ $(\ZZ/2)\setminus \ZZ$ of the dynamical
    point spectrum is not reflected by the diffraction spectrum of the
    DMS system, but can be recovered via the diffraction spectrum of a
    suitably chosen factor of it, either via the toy model of
    Section~$\ref{sec:toy}$ and Remark~$\ref{rem:simpler}$ or via the
    sytem $\YY$ from Proposition~$\ref{prop:factor-diff}$.
\end{theorem}

\begin{proof}
Most claims are clear from the previous propositions. The dynamical
spectrum for the action of $\RR$, which is written via the generating
elements (so that $\frac{1}{2}$ corresponds to an eigenvalue $-1$ as
explained earlier), follows from the spectrum for the $\ZZ$-action of
the shift in Proposition~\ref{prop:dms-dynamical} via a standard
suspension. Since we deal with a realisation of the system as a set of
Dirac combs with support $\ZZ$, this only extends the set $\{
\frac{1}{2}, 1 \}$ to the group generated by it, which is $\ZZ/2$.
\end{proof}

Up to this point, we can observe that the diffraction of $\XX$ does
not reflect the full dynamical spectrum of $\XX$, while the
diffraction of $\XX$ together with that of its factor $\YY$ does.

The analogous phenomenon appears in the case of the Thue-Morse system
\cite{EM,BG-TM}, where the dyadic rationals in the dynamical spectrum
\cite{Q} are only recovered via the diffraction spectrum of the period
doubling system, which is again a factor with pure point diffraction
(in fact, it can be described as a $2$-adic model set
\cite{BMS,BM,BG-PD}).  Let us thus look at this situation from a more
general point of view.

\section{General observations and outlook}

Here and below, $\XX$ is a compact dynamical system of (possibly
weighted) Dirac combs on $\ZZ^{d}$ or of translation bounded measures
on $\RR^{d}$, with ergodic invariant measure $\mu$ under the action of
the translation group $\ZZ^{d}$ or $\RR^{d}$. Let $\YY$ be a factor of
$\XX$, with factor map $\phi$ and induced measure $\nu$.  In
particular, we assume that diagram \eqref{eq:diagram} is again
commutative, with $S$ replaced by any generator of our translation
group.

If $g$ is an eigenfunction in $\mathrm{L}^{2} (\YY,\nu)$, it is clear
that $g \circ \phi \in \mathrm{L}^{2} (\XX,\mu)$, and the commutativity  
of the diagram \eqref{eq:diagram} implies that the latter is again
an eigenfunction, with the same eigenvalue (or set of eigenvalues, if
$d>1$).

\begin{fact}
   The dynamical eigenvalues of the factor system $\YY$ form a subset
   of those of the original system $\XX$.     \qed
\end{fact}

This is one ingredient for the following result; see
\cite{DGS,BL-2} for more.
\begin{prop}
   If\/ $\XX$ has pure point dynamical spectrum, then so does $\YY$.
   \qed
\end{prop}

More generally, it seems difficult for a factor to decrease the
long-range order in the diffraction, except for the removal of the
pure point part that corresponds to the trivial eigenfunction (via the
balanced weight representation). It is also rather clear that, by
means of suitable correlation functions, one can detect each
eigenfunction in the diffraction measure of a suitable factor.  This
is known explicitly for the Thue-Morse sequence, but also for the
Rudin-Shapiro sequence, both having the dyadic rationals as the pure
point part of the dynamical spectrum \cite{Q}.  Once a factor is pure
point, further factors can only reduce to subgroups, and hence do not
contain new information on the system. This mechanism also underlies
the equivalence of diffraction and dynamical spectrum in the pure
point case.

However, it is less obvious that a factor could display a singular
continuous diffraction spectrum if the original system does not. That
this is indeed possible is once again visible from the TM
sequences. Recall that the TM hull $\XX^{}_{\mathrm{TM}}$ can be
defined via the primitive substitution $1\mapsto 1 \bar{1}$,
$\bar{1}\mapsto \bar{1} 1$ on the binary alphabet $\{ 1, \bar{1}
\}$. For any $w\in \XX^{}_{\mathrm{TM}}$, replace $1$ and $\bar{1}$ by
the weights $\frac{1}{5}$ and $\frac{7}{5}$, followed by a random and
independent choice of a sign ($+$ or $-$) for each weight. This way,
one defines an (infinite) cover of $\XX^{}_{\mathrm{TM}}$. Each
element of it has average squared scattering strength $1$, while each
typical element has vanishing $2$-point correlations. Consequently,
the diffraction measure of the covering hull is $\widehat{\gamma} =
\lambda$, which is purely absolutely continuous. The TM system, which
is a factor, has purely singular continuous diffraction (for the
balanced weight case), while the period doubling system, which is
again a factor, is pure point. So, this little example illustrates a
step-wise unravelling of the order phenomena.

Note that what we say here is more general than (and somewhat
different from) the direct discussion of dynamical versus diffraction
spectrum in a single system. Indeed, if the diffraction spectrum is
pure point, then so is the dynamical spectrum. But it is certainly
possible to have a \emph{factor} with pure point diffraction spectrum
when the original system has a dynamical spectrum with also continuous
components -- this is what the known examples demonstrate. Somehow,
the dynamical spectrum contains the information of the diffraction
spectra of all its factors. Conversely, in all known examples so far,
the diffraction spectra of a system and its factors taken together
seem to comprise the complete information on the dynamical spectrum of
the original system (even though the diffraction spectrum of each
individual factor might not be very informative at all).

The general claim is rather clear now: The dynamical spectrum is not
to be compared with the diffraction spectrum of the system alone, but
with the diffraction spectra of the system and all its factors. A more
general and precise formulation and exposition is postponed to a
forthcoming publication \cite{BEL}.

Let us close by some remarks on the relation of our findings 
to some more general issues investigated in statistical physics.
Our DMS model forms a caricature of a system where `molecules' are
ordered, but due to a disordered interior of each molecule, `atoms' do
not display long-range order.  Although we have a 1-dimensional ground
state order for the dimeric molecules, which is typical for $T=0$, and
independent disorder on the `atomic' level, which is typical for
infinite temperatures, we expect conceptually similar phenomena to be
rather widespread. In more realistic models, one should have a similar
result for appropriate Gibbs measures, which then should be
higher-dimensional. For some preliminary results on diffraction,
mixing properties and spectra of equilibrium systems (as described
by Gibbs measures), we refer to \cite{BS,Kue1,Kue2,Sl}.

\smallskip
\section*{Acknowledgements}
It is a pleasure to thank Roberto Fern\'andez, Uwe Grimm, Daniel Lenz
and Frank Redig for helpful discussions.  We thank two reviewers for a number
of very useful suggestions. This work was supported by
the German Research Council (DFG), within the CRC 701.

\end{document}